\documentclass[english,11pt]{smfart}

\usepackage{etex}

\usepackage{amsbsy}
\usepackage{amsmath,amsfonts,amssymb,amsthm,mathrsfs,mathtools}

\usepackage{bm}

\usepackage[a4paper,vmargin={3cm,3cm},hmargin={3.5cm,3.5cm}]{geometry}
\usepackage[font=sf, labelfont={sf,bf}, margin=1cm]{caption}
\usepackage{graphicx}
\usepackage{epsfig}
\usepackage{latexsym}
\usepackage{xcolor}
\usepackage{ae,aecompl}
\usepackage{soul,framed}
\usepackage{comment}

\usepackage{xcolor}
\usepackage[pdfpagemode=UseNone,bookmarksopen=false,colorlinks=true,urlcolor=blue,citecolor=blue,citebordercolor=blue,linkcolor=blue]{hyperref}
\usepackage{smfhyperref}
\usepackage[capitalize]{cleveref}

\usepackage{pstricks}
\usepackage{enumerate}
\usepackage{tikz,animate,media9}						
\usepackage{todonotes}
\usepackage{pifont}
\usepackage{bm,marvosym}
\usepackage{algorithm}
\usepackage{algorithmic}

\usepackage{bbm}

\usepackage{tcolorbox}

\definecolor{aleacolor}{rgb}{0.16,0.59,0.78}

\hypersetup{
	breaklinks,
	colorlinks=true,
	linkcolor=aleacolor,
	urlcolor=aleacolor,
	citecolor=aleacolor}

\newcount\colveccount
\newcommand*\colvec[1]{
	\global\colveccount#1
	\begin{pmatrix}
		\colvecnext
	}
	\def\colvecnext#1{
		#1
		\global\advance\colveccount-1
		\ifnum\colveccount>0
		\\
		\expandafter\colvecnext
		\else
	\end{pmatrix}
	\fi
}

\newcommand{\ndR}{\mathbb{R}}

\renewcommand{\Pr}[1]{\mathbb{P}(#1)}

\newcommand{\Prb}[1]{\mathbb{P}\left(#1\right)}

\newcommand{\Ex}[1]{\mathbb{E}[#1]}

\newcommand{\Exb}[1]{\mathbb{E}\left[#1\right]}

\newcommand{\Va}[1]{\mathbb{V}[#1]}

\newcommand{\convdis}{\,{\buildrel d \over \longrightarrow}\,}

\newcommand{\convp}{\,{\buildrel p \over \longrightarrow}\,}

\newcommand{\convas}{\,{\buildrel a.s. \over \longrightarrow}\,}

\newcommand{\eqdist}{\,{\buildrel d \over =}\,}

\newcommand{\He}{\mathrm{H}}
\newcommand{\he}{\mathrm{h}}

\newcommand{\cB}{\mathcal{B}}
\newcommand{\cC}{\mathcal{C}}

\newcommand{\cH}{\mathcal{H}}

\newcommand{\cN}{\mathcal{N}}

\newcommand{\cP}{\mathcal{P}}

\newcommand{\cT}{\mathcal{T}}

\newcommand{\cZ}{\mathcal{Z}}

\newcommand{\mG}{\mathsf{G}}

\newcommand{\mN}{\mathsf{N}}

\newcommand{\mP}{\mathsf{P}}

\newtheorem{theorem}{Theorem}[section]

\newtheorem{proposition}[theorem]{Proposition}
\newtheorem{lemma}[theorem]{Lemma}

\numberwithin{equation}{section}


\title{\textbf{A branching process approach to level-$k$ phylogenetic networks}}
\date{}

\author{Benedikt Stufler}

\address[Benedikt Stufler]{Vienna University of Technology}
\email{benedikt.stufler at tuwien.ac.at}

\begin{document}

\vspace {-0.5cm}

\begin{abstract}
	The mathematical analysis of random phylogenetic networks via analytic and algorithmic methods has received increasing attention in the past years. In the present work we introduce branching process methods to their study. This approach appears to be new in this context. Our main results focus on random level-$k$ networks with $n$ labelled leaves. Although the number of reticulation vertices in such networks is typically linear in $n$, we prove that their asymptotic global and local shape is tree-like in a well-defined sense. We show that the depth process of vertices in a large network converges towards a Brownian excursion after rescaling by $n^{-1/2}$. We also establish Benjamini--Schramm convergence of large random level-$k$ networks towards a novel random infinite network.
\end{abstract}


\maketitle

\section{Introduction}

Phylogenetic networks may be used to model the evolutionary history of species that have undergone  reticulate events~\cite{bio1}, such as horizontal gene transfer (by which genes are transferred across species) or hybrid speciation (by which lineages recombine to create a new one)~\cite{bio2}.

The application of phylogenetic networks in  evolutionary biology motivates the mathematical study of their number and shape, which has received increasing attention in recent literature. See for example~\cite{fuchs2020asymptotic, bienvenu2020combinatorial, BGM, MR3900551, MR3319869, MR4109661, fuchs2020counting} and references given therein. In the present work, we introduce branching process methods to their study. This approach appears to be new in this context and makes a fine addition to the current toolbox of analytic and algorithmic methods. 


\subsection{Asymptotic enumeration of level $k$-networks}

A binary rooted phylogenetic network $N$ on a finite non-empty set $X$ of leaves  may be defined as a simple rooted directed graph with no directed cycles that satisfies the following constraints:
\begin{enumerate}
	\item It's unique \emph{root} has indegree $0$ and outdegree $2$.
	\item All non-root vertices are either \emph{tree nodes} (indegree $1$, outdegree $2$), \emph{reticulation nodes} (indegree $2$, outdegree $1$), or \emph{leaves} (indegree $1$, outdegree $0$).
\end{enumerate}
The leaves are bijectively labelled with elements of $X$. It will be notationally convenient to additionally admit the network consisting of a single labelled root vertex with no edges.

 There is an infinite number of such  networks on a given set $X$. For this reason, one restricts to subclasses for which this number is finite. We are going to focus on level-$k$ networks, with $k$ denoting a fixed  positive integer. Recall that a \emph{cutvertex} in a connected graph is a vertex whose removal disconnects the graph. Similarly, a \emph{bridge} is an edge whose removal disconnects the graph. A \emph{block} (or \emph{$2$-connected component}) is a connected induced subgraph that is maximal with the property of having no cutvertices of its own. The reader may consult books on the foundation of graph theory for further details~\cite{MR1633290}. We say the binary rooted phylogenetic network $N$  is a \emph{level-$k$} network, if the following conditions are met:
\begin{enumerate}
	\item Any block of $N$ contains at most $k$ reticulation vertices of $N$.
	\item Any block of $N$ with at least $3$ vertices contains at least $2$ vertices that are the source of bridges of $N$.
\end{enumerate}
Here we view directed edges as joining a \emph{source} vertex to a \emph{destination} vertex according to their orientation.  The second condition ensures that there are only finitely many such networks on a given set $X$. In fact, their number satisfies the following asymptotic expression:

\begin{lemma}
	\label{le:asymptotics}
	The number $N(k,n)$ of level-$k$ networks on an $n$-element set satisfies
	\begin{align}
		N(k,n) \sim a_k n^{-3/2} \rho_k^{-n} n! \qquad \text{as} \qquad  n \to \infty
	\end{align}
	for constants $a_k, \rho_k >0$ that only depend on $k$.
\end{lemma}
For $k=1$ and $k=2$ this was already shown  by~\cite{BGM} via analytic methods, who additionally calculated the involved constants,  gave exact enumerating formulas, studied unrooted networks, and proved limit theorems for the numbers of undirected cycles and inner edges in random networks.

\subsection{Global shape and limits of extremal parameters}

Note that each vertex $v$ in $N$ may be reached from the root by following a directed path that only traverses edges according to their orientation. There may be multiple such paths, and we denote the length of a shortest path by $\he_N(v)$. Often, $\he_N(v)$ is called the depth or height of $v$. The maximal height of vertices in $N$ is denoted by $\He(N)$. We let $|N|$ denote the number of vertices of~$N$.

Throughout this paper, we fix a positive integer $k$. For each integer $n \ge 2$ we let $\mN_n$ be drawn uniformly at random among all level-$k$ networks on the set $[n] := \{1, \ldots, n\}$. If we order its vertices $v_1, \ldots, v_{|\mN_n|}$ we may form the corresponding continuous height process $(\he_{\mN_n}({v_t}) : 0 \le t \le |\mN_n|)$ that starts at $\he_{\mN_n}(v_0) := 0$ and linearly interpolates the values $\he_N(v_i)$ for integers $i \in \{0, \ldots, |\mN_n|\}$. Of course, the height process depends on the order of vertices we choose.

\begin{theorem}
	\label{te:main1}
	We may couple $\mN_n$ with an ordering of its vertices such that the associated height process $(\he_{\mN_n}(v_t) : 0 \le t \le |\mN_n|)$ satisfies
	\begin{align}
		\label{eq:functional}
		(b_k n^{-1/2} \he_{\mN_n}(v_{s |\mN_n|}) : 0 \le s \le 1) \convdis (\mathsf{e}(s) : 0 \le s \le 1) \qquad \text{as} \qquad  n \to \infty
	\end{align}
	for a constant $b_k >0$ that only depends on $k$. Here $(\mathsf{e}(s) : 0 \le s \le 1)$ denotes Brownian excursion normalized to have duration $1$. The same limit as in~\eqref{eq:functional} holds if we form the height process using leaves only.
\end{theorem}

Questions concerning heights in various classes of random phylogenetic networks were raised in~\cite[Sec. 7]{MR3319869} and~\cite[Sec. 8]{fuchs2020asymptotic}.

We prove Theorem~\ref{te:main1} by establishing a new relation between distances in the random phylogenetic network $\mN_n$ and distances in a Galton--Watson tree conditioned on having $n$ leaves. This allows us to apply invariance principles for the latter given in~\cite{MR2946438, MR3335013}.


An immediate consequence of~Theorem~\ref{te:main1} is that 
\begin{align}
	\label{eq:height}
	b_k n^{-1/2} \He(\mN_n) \convdis \sup_{0 \le s \le 1} \mathsf{e}(s).
\end{align}
Likewise, a uniformly selected vertex (or a uniformly selected leaf) $u_n$ of $\mN_n$ satisfies
\begin{align}
	\label{eq:point}
	b_k n^{-1/2} \he_{\mN_n}(u_n) \convdis  \mathsf{e}(r)
\end{align}
with $0 \le r \le 1$ denoting a random point of the unit interval $[0,1]$, selected independently according to the uniform distribution. We may also obtain tail-bounds and convergence of moments:

\begin{theorem}
	\label{te:main2}
	There are constants $C,c>0$ such that for all $x>0$ and $n \ge 2$
	\begin{align}
		\label{eq:heightbound}
		\Pr{\He(\mN_n) > x} \le C \exp(-cx^2/n).
	\end{align}
	Moreover, all higher moments in~\eqref{eq:height} and~\eqref{eq:point} converge.
\end{theorem}

Our main tool for proving Theorem~\ref{te:main2} is a similar tail-bound by~\cite{MR3077536} for the height of  Galton--Watson trees conditioned to be large.  Theorem~\ref{te:main2} entails that the moments converge in~\eqref{eq:height}, yielding
\begin{align}
	\label{eq:moments1}
	b_k n^{-1/2} \Ex{\He(\mN_n)} \to \Exb{\sup_{0 \le s \le 1} \mathsf{e}(s)} =  \sqrt{\pi / 2}
\end{align}
and
\begin{align}
	\label{eq:moments2}
b_k^p n^{-p/2} \Ex{\He(\mN_n)^p}	\to 2^{-p/2} p(p-1) \Gamma(p/2) \zeta(p) 
\end{align}
for each integer $p \ge 2$. Here $\Gamma$ refers to Euler's gamma-function, and $\zeta$ to Riemann's zeta-function. Compare with~\cite[Eq. (1.6)]{MR3077536}.

Instead of using lengths of shortest directed paths, we may also study  the structure of the undirected graph $\mG_n$ underlying the network $\mN_n$. The graph distance from the root to a vertex in this graph needs not coincide with the length of a shortest \emph{directed} path. It might be shorter. If we write $\He(\mG_n)$ for the maximal height with respect to the graph distance, then by construction $\He(\mG_n) \le \He(\mN_n)$.  Hence, the tail bound of Theorem~\ref{te:main2} also applies to $\mG_n$:
\begin{align}
	\Pr{\He(\mG_n) > x} \le C \exp(-cx^2/n).
\end{align}
Moreover,  if we let $\he_{\mG_n}(v)$ denote the graph distance  of a vertex in $\mG_n$ from the root and construct the height process for $\mG_n$ accordingly, then it is clear from the proof of Theorem~\ref{te:main1} that there exists a constant $b_k' \ge b_k$ such that
	\begin{align}
	\label{eq:functionalprime}
	(b_k' n^{-1/2} \he_{\mG_n}(v_{s |\mG_n|}) : 0 \le s \le 1) \convdis (\mathsf{e}(s) : 0 \le s \le 1)
\end{align}
Likewise, the same convergence holds for the height process of the leaves.

Finally, we may also study the maximal length of a directed or undirected path starting from the root in $\mN_n$. (As opposed to the maximal length of a shortest directed or undirected path from the root to any vertex in $\mN_n$.) It is clear from the proofs of Theorems~\ref{te:main1} and \ref{te:main2} that~\eqref{eq:height},~\eqref{eq:heightbound},~\eqref{eq:moments1},~and~\eqref{eq:moments2}  hold analogously for these parameters, if we replace the constant $b_k$ by some constant $b_k''>0$ for the maximal length of a directed path, or by some constant $b_k'''>0$ for the maximal length of an undirected path.

The following theorem states that although pyhlogenetic networks are not trees, their global shape is tree-like:

\begin{theorem}
	\label{te:main3}
	Let $\mG_n$ be the graph underlying the random level-$k$ phylogenetic network~$\mN_n$, let $b_k'> 0$ be as in~\eqref{eq:functionalprime}. Let $\mu_n$ denote either the uniform measure on the vertices or on the leaves of $\mG_n$. Let $(\cT_{\mathsf{e}}, d_{\cT_{\mathsf{e}}}, \mu_{\cT_{\mathsf{e}}})$ denote the Brownian continuum random tree. Then
	\begin{align}
			(\mG_n, b_k' n^{-1/2} d_{\mG_n}, \mu_n) \convdis (\cT_{\mathsf{e}}, d_{\cT_{\mathsf{e}}}, \mu_{\cT_{\mathsf{e}}}) \qquad \text{as} \qquad  n \to \infty.
	\end{align}
in the Gromov--Hausdorff--Prokhorov sense.
\end{theorem}

The intuition behind Theorem~\ref{te:main3} is that although $\mG_n$ contains a linear number of cycles, the global shape is tree-like as the cycles are so short (at most $O(\log n)$ in circumference) that they contract to points when rescaling distances by $n^{-1/2}$. Care has to be taken that they nevertheless influence the global shape. A path between two typical points in $\mG_n$  has length roughly $\sqrt{n}$ and traverses roughly $\sqrt{n}$ cycles, thus they distort the distance on average by a stretch factor. 

The Brownian continuum random tree $(\cT_{\mathsf{e}}, d_{\cT_{\mathsf{e}}}, \mu_{\cT_{\mathsf{e}}})$ was introduced and studied in the series of pioneering papers~\cite{MR1085326, MR1166406, MR1207226}. In some sense, it's a random ``continuum''  tree with  uncountably many points. Formally, it may be defined as the random metric space $\cT_{\mathsf{e}}$ corresponding to the random semi-metric 
\begin{align}
	\label{eq:crtmetric}
	d(x,y) = \mathsf{e}(x) + \mathsf{e}(y) - 2\inf_{\min(x,y) \le t \le \max(x,y)} \mathsf{e}(t), \qquad x,y \in [0,1]
\end{align}
on the unit-interval. The measure $\mu_{\cT_{\mathsf{e}}}$ is the push-forward of the uniform measure on $[0,1]$ under the canonical surjection $[0,1] \to \cT_{\mathsf{e}}$.  The Brownian continuum random tree is universal in the sense that it describes the asymptotic geometry of a variety of different models of random graphs \cite{MR3382675, PaStWe2016,MR4132643,StEJC2018}, hence linking seemingly unrelated models of random structures. 

 The notion of Gromov--Hausdorff--Prokhorov convergence was introduced in~\cite{MR2571957}. It allows us to assign a distance to a pair $(X,d_X, \mu_X)$ and $(Y, d_Y, \mu_Y)$ of compact metric spaces equipped with Borel probability measures. We refer the reader to~\cite[Sec. 6]{MR2571957} for a detailed introduction.


\subsection{Local shape and additive parameters}

How does the vicinity of the root or a random vertex (or random leaf) in $\mN_n$ evolve as $n$ becomes large? The frequencies with which we observe given shapes converge, allowing us to describe the asymptotic local shape of $\mN_n$ via random networks having a countably infinite number of vertices.

In order to make this precise, we require some notation.
For any integer $\ell \ge 0$ we may consider the restriction $U_\ell(\cdot)$ that maps a pair of a  phylogenetic network $N$ and one of its vertices $v$ to the $\ell$-neighbourhood subnetwork $U_\ell(N,v)$ induced by all vertices that may be reached from $v$ via an undirected path of length at most $\ell$. That is, all vertices that may be reached from $v$ by crossing at most $\ell$ edges, regardless whether we follow the direction of the edges or not. We consider the vertices of $U_\ell(N,v)$ as unlabelled, and the edges as directed. If $v$ is equal to the root vertex of $N$ we may simply write $U_\ell(N)$.

\begin{theorem}
	\label{te:local1}
	There is a random infinite level-$k$ network $\hat{\mN}$ such that 
	\begin{align}
		\label{eq:local1}
		\mN_n \convdis \hat{\mN} \qquad \text{as} \qquad  n \to \infty
	\end{align}
	in the local topology. Even stronger, for any sequence $\ell_n$ of positive integers satisfying $\ell_n = o(\sqrt{n})$ it holds that
	\begin{align}
		\label{eq:sqqrt}
		d_{\textsc{TV}}(U_{\ell_n}(\mN_n), U_{\ell_n}(\hat{\mN})) \to 0.
	\end{align}
\end{theorem}

The limit~\eqref{eq:local1} states that for each integer $\ell \ge 0$ and each finite directed vertex-marked graph $G$ it holds that
\begin{align}
	\Pr{U_\ell(\mN_n) = G} \to \Pr{U_\ell(\hat{\mN}) = G}.
\end{align}
Thus,~\eqref{eq:sqqrt} is a much stronger statement, as it allows us to describe the asymptotic shape of larger neighbourhoods of the root. The assumption $\ell_n = o(\sqrt{n})$ is as general as possible: Thm.~\ref{te:main1} implies that for any constant but arbitrarily small $\epsilon>0$ the convergence in~\eqref{eq:sqqrt} fails to hold for $\ell_n = \lfloor \epsilon \sqrt{n} \rfloor$.

It is natural to also study  what happens in the vicinity of a uniformly selected vertex  of $\mN_n$.

\begin{theorem}
	\label{te:local2}
	Let $u_n$ denote a uniformly selected vertex of $\mN_n$. There is a random infinite directed vertex marked graph $\hat{\mN}^*$ such that 
	\begin{align}
		\label{eq:annealed}
		(\mN_n, u_n) \convdis \hat{\mN}^* \qquad \text{as} \qquad  n \to \infty
	\end{align}
	in the local topology.  Even stronger, for any sequence $\ell_n$ of positive integers satisfying $\ell_n = o(\sqrt{n})$ it holds that
	\begin{align}
		\label{eq:sqqrt2}
		d_{\textsc{TV}}(U_{\ell_n}(\mN_n, u_n), U_{\ell_n}(\hat{\mN}^*)) \to 0.
	\end{align}
	 Moreover, for any integer $\ell \ge 0$ and any finite directed unlabelled vertex-marked graph $G$ the  number $N_{\ell, G}$ of vertices $u$ in $\mN_n$ with $U_\ell(\mN_n, u) = G$ satisfies
	\begin{align}
		\label{eq:quenched}
		\frac{N_{\ell,G}}{|\mN_n|} \convp \Pr{U_\ell(\hat{\mN}^*) = G}.
	\end{align}
\end{theorem}

Borrowing terminology from statistical physics,~\eqref{eq:annealed} may be called annealed local convergence, and~\eqref{eq:quenched} quenched local convergence. 

We might also wonder what happens in the vicinity of a uniform random leaf $w_n$ of $\mN_n$. Theorem~\ref{te:local2} actually entails a local limit for the vicinity of $w_n$ as well. Indeed, it follows directly from~Theorem~\ref{te:local2} that 
\begin{align}
	\frac{|\mN_n|}{n} \convp \frac{1}{\Pr{\text{marked vertex of $\hat{\mN}^*$ is a leaf} } }.
\end{align}
Setting
\begin{align}
	\hat{\mN}^*_{\mathrm{l}} := (\hat{\mN}^* \mid \text{marked vertex of $\hat{\mN}^*$ is a leaf}),
\end{align}
it follows that 
\begin{align}
	(\mN_n, w_n) \convdis \hat{\mN}^*_{\mathrm{l}}
\end{align}
and even
\begin{align}
	d_{\textsc{TV}}(U_{\ell_n}(\mN_n, w_n), U_{\ell_n}(\hat{\mN}^*_{\mathrm{l}})) \to 0.
\end{align}
Likewise, it follows that the number $N_{\ell, G}^{\text{leaf}}$ of leaves whose $\ell$-neighbourhood equals $G$ satisfies
\begin{align}
		\frac{N_{\ell,G}^{\text{leaf}}}{n} \convp \Pr{U_\ell(\hat{\mN}^*_{\mathrm{l}}) = G }.
\end{align}

Using general principles for locally convergent random graphs~\cite{2015arXiv150408103K}, Theorem~\ref{te:local2} entails laws of large numbers for subgraph counts. That is, the number of copies of some given finite graph in $\mN_n$ concentrates at a constant multiple of $n$, with the constant factor being given by the expected value of some functional of $\hat{\mN}^*$. See~\cite[Lem. 4.3]{2015arXiv150408103K} for details.

\subsection*{Notation}

Unless otherwise stated, all unspecified limits are taken as $n \to \infty$.  The arrows $\convp$, $\convdis$, and $\convas$ denote convergence in probability, convergence in distribution, and almost sure convergence. Equality in distribution is denoted by $\eqdist$. The total variation distance between measures and random variables is denoted by~$d_{\textsc{TV}}$. For a vertex $v$ in a directed graph $N$ its \emph{indegree} refers to the number of edges in $N$ whose \emph{destination} is $v$. The \emph{outdegree} of $v$ is the numberof edges in $N$ whose \emph{source} is $v$. We will refer to a vertex $w$ as a \emph{child} of $v$ if there is a directed edge from $v$ to $w$ in $N$. In a rooted undirected tree we use the notions ``outdegree'' and ``child'' as if all edges were directed as pointing away from the root.

\section{A bijective encoding in terms of decorated trees}

Throughout the following, we fix an integer $k \ge 1$ and use the short term \emph{$k$-network} in order to refer  to binary rooted  phylogenetic level-$k$ networks.   The collection of $k$-networks on a given set $X$ will be denoted by $\cN[X]$. Recall that we admit the $k$-network consisting of a single labelled root vertex with no edges.

In \cite{MR2454413, MR2544376} a decompositions of level-$k$ networks into smaller simple networks was developed. In this section we explain how this leads to a blow-up procedure for the random generation of $k$-networks.

\subsection{Level-$k$ networks are blow-ups of decorated  trees}


Let $N$ be a $k$-network and let $v$ be a vertex of $N$ with outdegree $2$.  We say the vertex $v$ \emph{splits} (or is \emph{splitting}), if $N$ contains no vertex that may be reached via a directed path from both children of $v$. That is, we cannot walk from each of its two children to the same vertex by only crossing edges in accordance with their direction.  

This terminology allows us to differentiate three types of $k$-networks:
\begin{enumerate}
	\item The trivial $k$-network consisting of a single root vertex and nothing else.
	\item $k$-networks where the root has outdegree $2$ and splits.
	\item $k$-networks where the root has outdegree $2$ and does not split.
\end{enumerate}

It is easy to describe a $k$-network $N$ on a set $X$ where the root has outdegree $2$ and splits.  $N$ is obtained in a unique way by taking an unordered pair of $k$-networks with disjoint leaf sets that partition $X$, and adding a root vertex that is the source of two directed edges that point to the respective roots of these two $k$-networks.

Now, suppose that $N$ is a $k$-network where the root has outdegree $2$ and does not split. Then it is contained in a unique block $B$ with at least $4$ vertices. Vertices with indegree $1$ and outdegree $1$ in $B$ are tree nodes of $N$, and hence the source of a bridge of $N$ whose destination lies outside of $B$. Likewise, vertices of $B$ with indegree $2$ and outdegree $0$ are reticulation nodes of $N$, and the source of a bridge of $N$ whose destination lies outside of $B$. Let $S$ be the network obtained from $B$ by additionally adding these bridges and their destinations, which then correspond precisely to the leaves of $S$. We say $S$ is a \emph{simple network}, and $B$ is its \emph{core}. Thus, the network $N$ may be obtained from the $k$-network $S$  by identifying the leaves of $S$ with the  roots of smaller $k$-networks. The leaf sets of the $k$-networks attached to $S$ in this way partition the set $X$ into non-empty disjoint subsets. We may view the leaf set of a $k$-network attached to a leaf of $S$ as the label of the leaf. Thus, any $k$-network $N$ on a set $X$ whose root does not split may be obtained in a unique way by forming a partition $M$ of $X$ into non-empty subsets, choosing a simple $k$-network $S$ on $M$, and for each partition class $C \in M$ we identify the leaf of $S$ labelled by $C$ with the root of a $k$-network on $C$. 

Letting $N(k,n), B(k,n) \ge 0$ denote the  number of $k$-networks and  simple $k$-networks on a given $n$-element set, it follows by standard combinatorial tools~\cite{MR633783,MR1629341,MR2483235} that the exponential generating series
\[
	\cN(z) := \sum_{n \ge 1} N(k,n) \frac{z^n}{n!} \qquad \text{and} \qquad \cB(z) := \sum_{n \ge 2} B(k,n) \frac{z^n}{n!}
\]
satisfy the equation
\begin{align}
	\label{eq:decomposition}
	\cN(z) = z + \frac{\cN(z)^2}{2} + \cB(\cN(z)).
\end{align}

Moreover, it also follows that to each $k$-network $N$ on a set $X$ of at least two leaves we may associate a $k$-network $\mathrm{head}(N)$, which is either a simple network on some partition (with non-empty partition classes) of $X$, or it is the cherry network (consisting of a root with two children) on a $2$-partition of $X$ (with non-empty partition classes). The network $N$ is obtained by identifying the leaves of its head with the roots of the  subnetworks on the corresponding partition classes. 

Whenever one of these subnetworks is not the trivial network consisting of a single leaf, it has a head-structure of its own and is constructed from this head structure and even smaller subnetworks. We may proceed recursively in this way until only trivial subnetworks are left. In this way, we may form a pair $\Lambda(N) := (T, \delta)$ of a rooted unordered tree $T$ whose leaves are labelled bijectively with the element of $X$, and a function $\delta$ that assigns to each inner vertex $v$ of $T$ a head structure $\delta(v)$ as its \emph{decoration}. We say $(T, \delta)$ is a \emph{decorated tree}. The formal definition is as follows:
\begin{enumerate}
	\item If a network $N$ is trivial, that is it consists of a single leaf and nothing else, we let $T$ be given by a rooted tree consisting of a single vertex (labelled like the leaf of $N$). As $T$ has no inner vertices, $\delta$ is the trivial function with an empty domain.
	\item If a network $N$ has at least two leaves, it consists of the network $\mathrm{head}(N)$ with some number $\ell \ge 2$ of (possibly trivial) subnetworks $N_1, \ldots, N_\ell$ attached to the leaves of $\mathrm{head}(N)$. We let $T$ be the tree whose root $o$ has $\ell$ children, such that the $i$th child is the root of the tree corresponding (recursively) to $N_i$. We set $\delta(o) = \mathrm{head}(N)$ and extend $\delta$ according to the decorations in $\Lambda(N_1), \ldots, \Lambda(N_\ell)$.
\end{enumerate}

Thus:
\begin{lemma}
	\label{le:encoding}
	For each finite non-empty set $X$, the function $\Lambda$ ($= \Lambda_X$) is a bijection between the collection $\cN[X]$ of $k$-networks on $X$, and the collection $\cP[X]$ of pairs $(T,\delta)$, where $T$ is a rooted unordered tree whose leaves are bijectively labelled with the elements of~$X$, and $\delta$ is a function that assigns to each inner vertex $v$ of $T$ a head structure (that is, a cherry network or a simple network) on the children of $v$ (or equivalently the corresponding partition of $X$).
\end{lemma}

Lemma~\ref{le:encoding} identifies the combinatorial class of $k$-networks with a  special case of a class of Schr\"oder-enriched parenthesizations.  In general, if we have a class of combinatorial structures where each structure has a ``size'' given by a positive integer, we may form the corresponding class of Schr\"oder-enriched parenthesizations. It is the collection of all unordered rooted trees where each inner vertex has a decoration given by a structure whose size agrees with the outdegree of the vertex. See~\cite{MR1284403} for details. In the present case, the structures are the  head networks, and the ``size'' of a head network is its number of leaves.

The inverse $\Lambda^{-1}$ function of $\Lambda$ may be described as a  blow-up procedure, that maps a decorated tree $(T, \delta)$ to the network obtained by ``blowing up'' each inner vertex $v$ of $T$ by its decoration $\delta(v)$. That is, we delete the edges between $v$ and its children, identify the root of $\delta(v)$ with $v$, and identify each child of $v$ with the corresponding leaf in $\delta(v)$.

\subsection{A sampling procedure that uses simply-generated trees}

Let $\mP_n$ be uniformly selected from the collection $\cP[X]$ for $X :=\{1, \ldots, n\}$, $n \ge 2$.  Since, $\Lambda^{-1}: \cP[X] \to \cN[X]$ is a bijection, it follows that the $k$-network $\Lambda^{-1}(\mP_n)$ is distributed like the uniform $k$-network $\mN_n$ on $X$. Hence we may assume without loss of generality that
\begin{align}
		\mN_n = \Lambda^{-1}(\mP_n). 
\end{align}

Letting $H(k,i) \ge 0$ denote the number of head-networks with $i$ leaves for all $i \ge 0$, and set
\begin{align}
	\label{eq:defH}
		\cH(z) := \sum_{i \ge 2} \frac{H(k,i)}{i!} z^i = \cB(z) + z^2/2.
\end{align}
We refer to $\mathbf{w} := (H(k,i)/i!)_{i \ge 0}$ as a \emph{weight-sequence}. For each set $Y$ we let $\cH[Y]$ denote the collection of head-structures with leaves labelled bijectively by the elements of~$Y$. For ease of notation, we  set
\begin{align}
	\cH[d] := \cH[\{1, \ldots, d\}]
\end{align}
for all integers $d \ge 0$.

A \emph{(planted) plane tree} is a rooted ordered unlabelled tree. The children of any of its vertices are endowed with a linear order. We refer the reader to~\cite[Sec. 1.2.2]{MR2484382} for a detailed introduction to this type of trees. The following procedure, formulated in~\cite[Lem. 6.7]{MR4132643} for general Schr\"oder enriched parenthesizations, allows us to generate $\mP_n$ from a random (weighted) plane tree with $n$ leaves:

\begin{proposition}[{\cite[Lem. 6.7]{MR4132643}}]
	Set $p_0 = 1$ and $p_i = H(k,i) / i!$ for all $i \ge 1$. The outcome $(\tau_n, \delta_n)$ of the following procedure is distributed like $\mP_n$.
	\begin{enumerate}
		\item Generate a random plane tree $\tau_n$ with $n$ leaves according to its distribution
		\begin{align}
			\label{eq:dtau}
		\Pr{\tau_n = T} = (\sum_{P} \prod_{v \in P} p_{d_P^+(v)})^{-1} \prod_{v \in T} p_{d_T^+(v)},
		\end{align}
		with the sum-index $P$ ranging over the finite collection of plane trees with $n$ leaves, where each internal vertex has outdegree at least two.
		\item For each inner vertex $v$ of $\tau_n$ sample a head-structure 
		\begin{align} 
			\delta_n(v) \in \cH[d_{\tau_n}^+(v)] 
		\end{align} uniformly at random.
		\item Choose a bijection $\sigma$ between the  set of leaves of $\tau_n$ and $X$ uniformly at random, and distribute labels to the leaves of $\tau_n$ accordingly.
	\end{enumerate}
\end{proposition}

Note that the decorated tree $\mP_n$ is unordered, whereas in $(\tau_n, \delta_n)$ the children of any vertex are endowed with a linear order. If we forget about these linear orders, we obtain a decorated unordered tree that is distributed like $\mP_n$.  Note also that the distribution of $\tau_n$ in~\eqref{eq:dtau} depends on the weight sequence~$\mathbf{w}$. 

\section{Asymptotic enumeration}

In this section we will prove Lemma~\ref{le:asymptotics} and identify $\tau_n$ as a critical Galton--Watson tree conditioned on having $n$ leaves.

The number of possible simple components in  $k$-networks is infinite. However,  \cite{MR2454413} developed a decomposition of simple networks in terms of  so-called \emph{generators}. A generator is a directed multi-graph obtained from the core of a simple network by contracting each vertex of the core with in-degree $1$ and outdegree $1$. Conversely, simple networks are precisely the networks obtained by blowing up edges of generators into paths (that is, replacing the edge by a path with at least one edge), and afterwards adding outgoing edges to all vertices with indegree $1$ and outdegree $1$, and all vertices with indegree $2$ and outdegree $0$. However, care has to be taken  when a generator contains multi-edges, that is, pairs of vertices joined by exactly two edges. In this case,  for each such pair we have to blow-up at least one of the two edges by a path with at least two edges.

For all $i,j, \ell \ge 0$ let $G(k,i,j,\ell)$ denote the number of generators with the following properties:
\begin{enumerate}
	\item The number of vertices with indegree $2$ is at most $k$.
	\item There are $i$ vertices with outdegree $0$, labelled from $1$ to $i$.
	\item There are $j$ edges that are not multi-edges, labelled from $1$ to $j$.
	\item There are $\ell$ pairs of multi-edges joining the same two vertices, and each pair carries a unique label from $1$ to $\ell$.
\end{enumerate}

It follows that the series $\cH(z)$ defined in Equation~\eqref{eq:defH} satisfies
\begin{align}
	\label{eq:master1}
	\cH(z) = \frac{z^2}{2} + \sum_{i,j,\ell \ge 0} \frac{G(k,i,j,\ell)}{i! j! \ell!} z^i \frac{1}{(1-z)^j} \left( \frac{z}{1-z} + \frac{ z^2}{2(1-z)^2} \right)^\ell.
\end{align}
The total number of generators that contain at most $k$ vertices with indegree $2$ is finite, see~\cite{MR2454413, MR2544376}. Consequently, it follows from~\eqref{eq:master1} that there is a bivariate polynomial $f(z,w) \ne 0$ with $f(0,w)=0=f(z,0)$ such that
\begin{align}
	\cH(z) = \frac{z^2}{2} + f(z, (1-z)^{-1}).
\end{align}
This crucial equation entails that  $\cH(z)$ has radius of convergence $1$ and 
\begin{align}
	\lim_{t \nearrow 1} \cH'(t) = \infty.
\end{align}
Let $0<t_0<1$ be the unique point with $\cH'(t_0) = 1$. Let $\xi$ be a random non-negative integer with distribution given by
\begin{align}
	\Pr{\xi=\ell} = \begin{cases} p_\ell t_0^{\ell-1} & \ell \ge 1 \\ 1 - \sum_{i \ge 1} p_i t_0^{i-1} & \ell=0 \end{cases}.
\end{align}
By the choice of $t_0$, this is well-defined and 
\begin{align}
	\label{eq:crit1}
	\Ex{\xi} = \cH'(t_0) = 1.
\end{align}
As $t_0<1$, it follows that $\xi$ has finite exponential moments. That is, there is an $\epsilon>0$ such that
\begin{align}
	\label{eq:finexp}
	\Ex{  (1+\epsilon)^\xi  } < \infty.
\end{align}
Recall that a $\xi$-Galton--Watson $\tau$ is a random plane tree that starts with a single root vertex and where each vertex receives offspring according to an independent copy of $\xi$. Thus, if $T$ is a plane tree with $n$ leaves,
\begin{align}
	\Pr{\tau= T} = \prod_{v \in T} \Pr{\xi = d^+_T(v)} = \Pr{\xi=0}^n t_0^{n-1} \prod_{v \in T} p_{d^+_T(v)}.
\end{align}
Thus, $\Pr{\tau= T}$ is equal to the product of $\Pr{\tau_n=T}$ and a factor that only depends on $n$ (and not on $T$). Hence, $\tau_n$ is distributed like the result of conditioning $\tau$ on the event that its number $L(\tau)$ of leaves  is equal to $n$. That is,
\begin{align}
	\label{eq:condgwt}
	\tau_n \eqdist (\tau \mid L(\tau) = n).
\end{align}
Equations~\eqref{eq:crit1} and~\eqref{eq:finexp} allow us to apply a general result \cite[Thm. 3.1]{MR2544376} for the number of leaves in a critical Galton--Watson tree, yielding
\begin{align}
	\label{eq:heyoo}
	\Pr{L(\tau) = n} \sim \sqrt{\frac{\Pr{\xi=0}}{2 \pi \Va{\xi}}} n^{-3/2}
\end{align}
The probability generating function  $\cZ(z) := \Ex{z^{L(\tau)}}$
for $L(\tau)$ satisfies the recursive equation
\begin{align}
	\label{eq:ZZ1}
	\cZ(z) = z \Pr{\xi=0} + \sum_{\ell \ge 1} \Pr{\xi=\ell} \cZ(z)^\ell 
		= z (1- \cH(t_0)/t_0) + \cH(t_0\cZ(z)) / t_0. 
\end{align}
Equation~\eqref{eq:crit1} entails that 
\begin{align}
	t_0 = \sum_{i \ge 1} p_i t_0^i > \cH(t_0).
\end{align}
Hence it follows from Equation~\eqref{eq:ZZ1} that
\begin{align}
	t_0 \cZ(z / (t_0 - \cH(t_0))) = z + \cH(t_0  \cZ(z / (t_0 - \cH(t_0)))).
\end{align}
Equation~\eqref{eq:decomposition} entails that $\cN(z) = z + \cH(\cN(z))$, hence
\begin{align}
	\cN(z) = t_0 \cZ(z / (t_0 - \cH(t_0))).
\end{align}
In other words, it holds for all $n \ge 1$
\begin{align}
	N(k,n) =  \Pr{L(\tau)=n} t_0 (t_0 - \cH(t_0))^{-n}.
\end{align}	
Using~\eqref{eq:heyoo}, it follows that
\begin{align}
	N(k,n)		\sim  \sqrt{\frac{\Pr{\xi=0}}{2 \pi \Va{\xi}}} t_0 n^{-3/2}(t_0 - \cH(t_0))^{-n}
\end{align}
This proves Lemma~\ref{le:asymptotics} for $a_k := \sqrt{\frac{\Pr{\xi=0}}{2 \pi \Va{\xi}}} t_0$ and $\rho_k =  t_0 - \cH(t_0)$.

\section{The asymptotic global shape}

In the present section we prove Theorems~\ref{te:main1},~\ref{te:main2}, and~\ref{te:main3}.

\subsection{Preparations}

The following observation is easy, but we are going to use it often enough to state it explicitly:
\begin{proposition}
	\label{pro:bound}
	Let $d \ge 2$ be an integer. Then the number of vertices $|H|$ in a head structure $H \in \cH[d]$ is bounded by $2(d+k)$.
\end{proposition}
\begin{proof}
	This follows from the fact that any simple network has at most $k$ reticulation vertices, and all other vertices either have outdegree $2$ or $0$: 
	
	We may build  any simple network from a cherry network  (consisting of a root vertex with two children) step by step, where in each step we either add two children to a leaf, or fuse two leaves  together and add a single child to the newly created reticulation vertex. (Please note that not every network created in this way is a simple network.) The first kind of step increases the total number of vertices by two and increases the number of leaves by $1$. The second kind of step leaves the total number of vertices invariant and decreases the number of leaves by $1$. Letting $s_1$ and $s_2$ denote the number of steps of kind $1$ and $2$, the resulting network has $V := 3 + 2s_1$ vertices and $L := 2 + s_1 - s_2$ leaves. Since we can have at most $k$ reticulation vertices, $s_2 \le k$ holds, and hence 
	\[
	2(L+k) \ge 2(2+ s_1) \ge V.
	\]
	Since a head structure from $\cH[d]$ has $d$ leaves, it follows that it has at most $2(d+k)$ vertices.
\end{proof}

We are also going to need a bound for the maximal outdegree of $\tau_n$:
\begin{proposition}
	\label{pro:max}
	There is a constant $C>0$ such that the maximal outdegree $\Delta(\tau_n)$ of $\tau_n$ satisfies
	\begin{align}
		\Delta(\tau_n) \le C \log n
	\end{align}
with a probability that tends to $1$ as $n \to \infty$.
\end{proposition}
\begin{proof}
	Since $\Pr{\xi=1}=0$, it follows that $\tau_n$ has (almost surely) no inner vertices with outdegree $1$. Consequently, the number $|\tau_n|$ of vertices of $\tau_n$ is bounded by the number of vertices of a binary tree with $n$ leaves, that is,
	\begin{align}
		\label{eq:numvert}
		|\tau_n| \le 2n-1.
	\end{align}
Using~\eqref{eq:heyoo}, it follows that for any $x>0$
	\begin{align}
		\label{eq:oida1}
		\Pr{\Delta(\tau_n) > x} &= \frac{\Pr{L(\tau)=n, \Delta(\tau)>x}}{\Pr{L(\tau) = n}} \\
		&\le \Pr{L(\tau) = n}^{-1} 2n \Pr{\xi > x} \nonumber \\
		&\le O(n^{5/2}) \Pr{\xi > x}. \nonumber 
	\end{align}
	By~\eqref{eq:finexp} it follows that there are constants $C,c>0$ such that
	\begin{align}
		\label{eq:oida2}
		\Pr{ \xi > x} \le C \exp(-cx)
	\end{align}
	for all $x>0$. Taking $x = C' \log n$ for a sufficiently large constant $C'>0$, it follows that
	\begin{align}
		\Pr{\Delta(\tau_n) >  C' \log n} \to 0
	\end{align}
as $n \to \infty$.
\end{proof}

\subsection{Size-biased trees}
\label{sec:sizebias}

 Since $\Ex{\xi}=1$, we may define the size-biased random positive integer $\hat{\xi}$ by
\begin{align}
	\Pr{\hat{\xi} = i} = i \Pr{\xi=i}.
\end{align}
For each integer $\ell \ge 0$, we define the size-biased Galton--Watson tree $\hat{\tau}^{(\ell)}$ as a random finite plane tree with a marked vertex having height $\ell$. For $\ell =0$, we let $\hat{\tau}^{(0)}$ denote an independent copy of the Galton--Watson tree with a marked vertex that coincides with its root vertex. Inductively, we define $\hat{\tau}^{(\ell)}$ for $\ell \ge 1$ as follows: Start with a root vertex that receives offspring according to an independent copy of $\hat{\xi}$. A  child of the root is selected uniformly at random and identified with an independent copy of $\hat{\tau}^{(\ell-1)}$. Each of the remaining children of the root gets identified with an independent copy of $\tau$, that is, a fresh independent copy for each remaining child. 

Thus, the vertices of $\hat{\tau}^{(\ell)}$ that receive offspring according to $\hat{\xi}$ (as opposed to $\xi$) together with the marked vertex of $\hat{\tau}^{(\ell)}$ form a path of length $\ell$ from the root to the marked vertex of $\hat{\tau}^{(\ell)}$.

It is elementary to verify that for each finite plane tree $T$ and each vertex $v$ of $T$ with height $\ell$ it holds that
\begin{align}
	\label{eq:sz1}
	\Prb{\hat{\tau}^{(\ell)} = (T,v)} = \prod_{u \in T} \Prb{\xi= d_T^+(u)} = \Pr{\tau = T}.
\end{align}
Any locally finite rooted tree $T$ may be decorated in a \emph{canonical} random way by choosing for each inner vertex $v$ a decoration from $\cH[d_T^+(v)]$ uniformly at random, independently from the remaining decorations. Thus we may form canonical random decorations $(\tau, \delta)$ and $(\hat{\tau}^{(\ell)}, \hat{\delta}^{(\ell)})$. It follows from~\eqref{eq:sz1} that for any deterministic decoration $\gamma$ of $T$
\begin{align}
	\label{eq:sz2}
	\Prb{(\hat{\tau}^{(\ell)}, \hat{\delta}^{(\ell)}) = (T, \gamma)} &= \Pr{\tau= T} \prod_{\substack{u \in T,  d_T^+(u) > 0}} \frac{1}{|\cH[d^+_T(u)]|} \\
	&= \Prb{(\tau,\delta) = (T,\gamma)}. \nonumber
\end{align}

\subsection{Heights in random networks}

We let the height $\he_{N}(v)$ of a vertex $u$ in a network $N$ be the length of a shortest directed path from the root of $N$ to $u$.
Using the decorated trees $(\tau, \delta)$ and $(\hat{\tau}^{(\ell)}, \hat{\delta}^{(\ell)})$, we are now ready to prove the following:

\begin{lemma}
	\label{le:stretch}
	Let $\eta$ denote the length of a shortest directed path from the root of a uniformly selected network from $\cH[\hat{\xi}]$ to a leaf that is uniformly selected among its $\hat{\xi}$ leaves. With a probability that tends to $1$ as $n \to \infty$, any vertex $v$ of $\tau_n$ has the property, that the corresponding vertex $u$ in $\mN_n= \Lambda^{-1}(\mP_n)$ satisfies
	\begin{align}
		\label{eq:diss}
		|\he_{\mN_n}(u) - \Ex{\eta} \he_{\tau_n}(v)| \le n^{3/8}.
	\end{align}
\end{lemma}
\begin{proof}
	Let us take a closer look at the network $\Lambda^{-1}(\hat{\tau}^{(\ell)}, \hat{\delta}^{(\ell)})$. The path $v_0, \ldots, v_\ell$ from the root $v_0$ of $\hat{\tau}^{(\ell)}$ to its marked vertex $v_\ell$ corresponds to vertices $u_0, \ldots, u_\ell$ in $\Lambda^{-1}(\hat{\tau}^{(\ell)}, \hat{\delta}^{(\ell)})$. In particular, $u_0$ coincides with the root of $\Lambda^{-1}(\hat{\tau}^{(\ell)}, \hat{\delta}^{(\ell)})$. Every path in $\Lambda^{-1}(\hat{\tau}^{(\ell)}, \hat{\delta}^{(\ell)})$ from $u_0$ to $u_\ell$  must pass through $u_1, \ldots, u_{\ell-1}$ and is entirely contained in the subnetworks corresponding to $\hat{\delta}^{(\ell)}(v_0), \ldots, \hat{\delta}^{(\ell)}(v_{\ell-1})$. The length of a shortest directed path from $u_0$ to $u_\ell$ is distributed like the sum $\eta_1+ \ldots + \eta_\ell$ of~$\ell$ independent copies $\eta_1, \ldots, \eta_\ell$ of $\eta$.
	
	It follows from Equation~\eqref{eq:finexp} and Proposition~\ref{pro:bound} that $\eta$ has finite exponential moments, that is, 
	\begin{align}
		\label{eq:etaexp}
		\Ex{  (1+\epsilon)^\eta  } < \infty.
	\end{align}
	for some $\epsilon>0$. Let $\tilde{\eta} := \eta - \Ex{\eta}$. It follows that there is a constant $c>0$ such that for all sufficiently small $\lambda>0$
	\begin{align}
		\Ex{\exp(\lambda \tilde{\eta})} \le 1 + c \lambda^2 \qquad \text{and} \qquad \Ex{\exp(-\lambda \tilde{\eta})} \le 1 + c \lambda^2.
	\end{align}
	Using Markov's inequality, it follows that for all $x>0$
	\begin{align}
		\Pr{\eta_1 + \ldots + \eta_\ell - \ell \Ex{\eta}>x} &\le \Prb{ \exp \left( \lambda \sum_{i=1}^\ell (\eta_i - \Ex{\eta}) \right) > \exp(\lambda x)} \\ 
		& \le 
		\frac{\Exb{\exp(\lambda \tilde{\eta})}^\ell}{\exp(\lambda x)} \nonumber \\
		&\le \frac{(1 + c \lambda^2)^\ell }{\exp(\lambda x)} \nonumber.
	\end{align}
Repeating the same argument for $-\tilde{\eta}$ instead of $\tilde{\eta}$, we arrive at
\begin{align}
	\label{eq:llld}
	\Pr{ |\eta_1 + \ldots + \eta_\ell - \ell \Ex{\eta} |>x} \le 2 \frac{(1 + c \lambda^2)^\ell }{\exp(\lambda x)}.
\end{align}
	Taking $x = n^{3/8}$ and $\lambda= n^{-1/4}$, it follows that uniformly for all $1 \le \ell \le \sqrt{n} \log n$
	\begin{align}
		\label{eq:ld}
		\Pr{|\eta_1 + \ldots + \eta_\ell - \ell \Ex{\eta}|> n^{3/8}} \le \exp( - \Theta(n^{1/8}))
	\end{align}
	for all $\ell \ge 1$.
	
	The main result of~\cite{MR2946438} entails that the height of the tree $\tau_n$ admits a distributional limit when rescaled by $n^{-1/2}$. In particular, it is smaller than $\sqrt{n} \log n$ with a probability that tends to $1$ as $n \to \infty$. Thus, the probability that a ``bad'' vertex exists in $(\tau_n, \delta_n)$ such that Inequality~\eqref{eq:diss} fails is bounded by
	\begin{align}
		\label{eq:tobound1}
		o(1) + \sum_{\ell=1}^{\sqrt{n} \log n} \Pr{ (\tau_n, \delta_n) \text{ contains a ``bad'' vertex with height $\ell$} }.
	\end{align}
	We know by~\eqref{eq:numvert} that $\tau_n$ has at most $2n$ vertices in total. Critically applying Equation~\eqref{eq:sz2} and using Inequality~\eqref{eq:ld} and Equation~\eqref{eq:heyoo}, it follows that the sum in~\eqref{eq:tobound1} may be bounded by
	\begin{align}
		\label{eq:getbound}
		&\Pr{L(\tau)=n} \sum_{\ell=1}^{\sqrt{n} \log n} \Pr{ (\tau, \delta) \text{ contains a ``bad'' vertex with height $\ell$} }  \\
		&\le O(n^{3/2})  \sum_{\ell=1}^{\sqrt{n} \log n} (2n) \Pr{ \text{the marked vertex of } (\hat{\tau}^{(\ell)}, \hat{\delta}^{(\ell)}) \text{ is bad} }  \nonumber \\
		&\le O(n^{5/2}) \sum_{\ell=1}^{\sqrt{n} \log n} \Pr{|\eta_1 + \ldots + \eta_\ell - \ell \Ex{\eta}|> n^{3/8}}. \nonumber
	\end{align}
	By Inequality~\eqref{eq:ld}, it follows that this upper bound may be further bounded by $O(n^3 \log n) \exp(- \Theta(n^{1/8}) )$, which tends to zero as $n \to \infty$.  This completes the proof.
\end{proof}

For ease of notation, in everything that follows we will simply consider a vertex $v \in \tau_n$ of the tree $\tau_N$ also as a  vertex $v \in \mN_n$ of the network $\mN_n$. This saves us from repeatedly writing ``the vertex of $\mN_n$ that corresponds to $v \in \tau_n$''.

Note that $\mN_n$ is likely to have more vertices than $\tau_n$. Each head structure $\delta_n(v)$ for $v \in \tau_n$ may contribute a surplus of vertices that do not correspond to any vertices of $\tau_n$. These are precisely the non-root non-leaf vertices of $\delta_n(v)$. We will need fine-grained information on the growth of the surplus:

\begin{lemma}
	\label{le:spread}
	Let $v_1, \ldots, v_{|\tau_n|}$ denote the depth-first-search ordered list of vertices of $\tau_n$. For each head structure $H$ let $S(H)$ denote the number of surplus vertices of $H$, that is, the number of non-root non-leaf vertices of $H$. Let $\kappa$ denote the number of surplus vertices of a uniform random head structure from $\cH[\xi]$. With a probability that tends to $1$ as $n \to \infty$, 
	\begin{align}
		\label{eq:dif}
		\left| \sum_{i=1}^\ell S(\delta_n(v_i)) - \ell \Ex{\kappa} \right| \le n^{3/4}
	\end{align}
	holds for all $1 \le \ell \le |\tau_n|$.
\end{lemma}
\begin{proof}
	Note that $\kappa$ has finite exponential moments: $\xi$ has finite exponential moments by~\eqref{eq:finexp} and any head structure from $\cH[\xi]$ has at most $2(\xi + k)$ vertices by Proposition~\ref{pro:bound}. Hence $\kappa \le 2(\xi + k)$, entailing that
	\begin{align}
		\label{eq:finkappa}
		\Ex{(1 + \epsilon)^\kappa} < \infty
	\end{align}
	for some $\epsilon>0$.
	
	Let $\kappa_1, \kappa_2, \ldots$ denote independent copies of $\kappa$. The finite exponential moments property~\eqref{eq:finkappa} allows us to argue analogously as for Inequality~\eqref{eq:llld}, yielding that there is a constant $c>0$ such that for all sufficiently small $\lambda >0$ and all $x>0$ and integers $\ell \ge 1$
	\begin{align}
		\Pr{|\kappa_1 + \ldots + \kappa_\ell - \ell \Ex{\kappa}|>x}  \le 2\frac{(1 + c \lambda^2)^\ell }{\exp(\lambda x)}.
	\end{align}
		Taking $x=n^{3/4}$ and $\lambda= n^{-1/2}$, it follows that uniformly for all $1 \le \ell \le 2n$
	\begin{align}
		\label{eq:finineq}
		\Pr{|\kappa_1 + \ldots + \kappa_\ell - \ell \Ex{\kappa}|>n^{3/4}} \le \exp(-\Theta(n^{1/4})).
	\end{align}
	
	Recall that $\tau_n$ has at most $2n$ vertices by Inequality~\eqref{eq:numvert}. Hence the probability that there is a ``bad'' integer $1 \le \ell \le |\tau_n|$ for which~\eqref{eq:dif} fails is bounded by
	\begin{align}
		\label{eq:upperbound1}
		\Pr{L(\tau) = n}^{-1} \sum_{m=n}^{2n} \sum_{\ell=1}^m \Pr{|\tau|=m, \text{$\ell$ is ``bad''} }
	\end{align}
	Let $\xi_1, \xi_2, \ldots$ denote independent copies of $\xi$, and for each $i \ge 1$ let $H(\xi_i)$ be uniformly selected from $\cH[\xi_i]$. Thus, $S(H(\xi_1)), S(H(\xi_2)), \ldots$ are independent copies of~$\kappa$. 

	The depth-first-search ordered outdegrees of $\tau$ may be described by $(\xi_1, \ldots, \xi_L)$ with $L$ denoting the first integer for which $\sum_{i=1}^L (\xi_i -1) = -1$. The decorations $(\delta(v))_{v \in \tau}$ may be described by $H_1, H_2, \ldots$. Thus the  event  that $|\tau|=m$ and that $\ell$ is ``bad'' implies that
	\begin{align}
			\left| \sum_{i=1}^\ell S(H_i) - \ell \Ex{\kappa} \right| > n^{3/4}
	\end{align}
	Using~\eqref{eq:finineq} and~\eqref{eq:heyoo}, it follows that the upper bound~\eqref{eq:upperbound1} may be bounded further by
	\begin{align}
		\Theta(n^{3/2}) 2n^2 \exp(-\Theta(n^{1/4})).
	\end{align}
	This bound tends to zero as $n \to \infty$. Hence the proof is complete.
\end{proof}
We are now ready to prove Theorems~\ref{te:main1},~\ref{te:main2}, and~\ref{te:main3}.
\begin{proof}[Proof of Thm.~\ref{te:main1}]
	Let $v_1, \ldots, v_{|\tau_n|}$ denote the depth-first-search ordered list of vertices of $\tau_n$, which we will also consider as vertices of $\mN_n= \Lambda^{-1}(\mP_n)$. For each $1 \le i \le |\tau_n|$ let $s_i$ denote the size of the surplus $S(\delta_n(v_i))$ and let $v_{i,1}, \ldots, v_{i, s_i}$ denote the vertices of the surplus. Thus,
	\begin{align}
		\label{eq:specod}
		(u_1, \ldots, u_{|\mN_n|}) := (v_1, v_{1,1}, \ldots, v_{1,s_1}, \quad \ldots, \quad   v_{|\tau_n|}, v_{{|\tau_n|},1}, \ldots, v_{{|\tau_n|},s_{|\tau_n|}})
	\end{align} 
	is an ordering of the vertices of $\mN_n$. 
	
	There is a constant $C>0$ such that with a probability that tends to $1$ as $n \to \infty$, it holds for all $1 \le i \le |\tau_n|$ and all $1 \le j \le s_i$ that
	\begin{align}
		\label{eq:yvar}
		|\he_{\mN_n}(v_i) - \he_{\mN_n}(v_{i,j})| \le C \log n.
	\end{align}
	To see this, note that the difference of distances in~\eqref{eq:yvar} is bounded by the number of vertices $|\delta_n(v_i)|$ of the head structure $\delta_n(v_i)$.   By Proposition~\ref{pro:bound}, this is at most $2(d_{\tau_n}^+(v_i) +k)$. by Proposition~\ref{pro:max} this may be further bounded by
	\begin{align*}
		2(d_{\tau_n}^+(v_i) +k) \le 2(\Delta(\tau_n) + k) = O(\log n).
	\end{align*}
	Hence~\eqref{eq:yvar} holds with high probability for all $i$ and $j$.
	
	Note further that for each $1 \le r \le |\mN_n|$ there is a unique index $1 \le i(r) \le |\tau_n|$ such that either $u_r = v_{i(r)}$ or $u_r = v_{i(r),j}$ for some $1 \le j \le s_{i(r)}$. By Inequality~\eqref{eq:yvar} it follows that
	\begin{align}
		\label{eq:hahaha}
		|\he_{\mN_n}(u_r) - \he_{\mN_n}(v_{i(r)})| \le C \log n
	\end{align}
	holds for all $1 \le r \le |\tau_n|$ with a probability that tends to $1$ as $n \to \infty$. By Lemma~\ref{le:stretch} and the triangle inequality, it follows that
	\begin{align}
		\label{eq:AA}
		|\he_{\mN_n}(u_r) - \Ex{\eta}\he_{\tau_n}(v_{i(r)})| \le O(n^{3/8}).
	\end{align}
	
	Furthermore, Lemma~\ref{le:spread} (where $S(\delta_n(v_i))$ corresponds to the notation $s_i$ here) implies that with probability tending to $1$ as $n\to\infty$
	\begin{align}
		\label{eq:BB}
		| |\mN_n| - (1 + \Ex{\kappa})|\tau_n| | \le n^{3/4}
	\end{align}
	and for all $1 \le r \le |\mN_n|$
	\begin{align}
		\label{eq:CC}
		|r - (1 + \Ex{\kappa}) i(r) | \le O(n^{3/4}).
	\end{align}

	By the main result of~\cite{MR2946438}, there is a constant $b>0$ such that the height process $(\he_{\tau_n}(v_{t |\tau_n|}) : 0 \le t \le 1)$  (with linear interpolation between values $\he_{\tau_n}(v_i)$ for $1 \le i \le |\tau_n|$) satisfies
	\begin{align}
		\label{eq:DD}
		(b n^{-1/2}\he_{\tau_n}(v_{t |\tau_n|}) : 0 \le t \le 1) \convdis (\mathsf{e}(s) : 0 \le s \le 1)
	\end{align}
	as random elements of the space of continuous function $\cC([0,1], \ndR)$.

	Now it comes all together: By~\eqref{eq:AA},~\eqref{eq:BB},~\eqref{eq:CC}, and~\eqref{eq:DD} (and the fact $n \le |\tau_n| \le 2n$ from~\eqref{eq:numvert}) it follows that
	\begin{align}
		(b \Ex{\eta}^{-1}) n^{-1/2}\he_{\mN_n}(v_{t |\mN_n|}) : 0 \le t \le 1) \convdis (\mathsf{e}(s) : 0 \le s \le 1).
	\end{align}
	Thus, Equation~\eqref{eq:functional} holds with 
	\begin{align}
		b_k := b / \Ex{\eta}.
	\end{align}

	Let $w_1, \ldots, w_n$ denote the depth-first-search ordered list of leaves of $\tau_n$. Thus, $w_1, \ldots, w_n$ also correspond to the leaves of $\mN_n$. The  convergence
	\begin{align}
		(b \Ex{\eta}^{-1}) n^{-1/2}\he_{\mN_n}(w_{t n}) : 0 \le t \le 1) \convdis (\mathsf{e}(s) : 0 \le s \le 1).
	\end{align}
	follows by analogous arguments, since the number of vertices between consecutive leaves in the depth-first-search ordered list of vertices of $\tau_n$ may be controlled in an entirely analogous fashion as the surplus of the head structures, ensuring that
	\begin{align}
		(b  n^{-1/2}\he_{\tau_n}(w_{t n}) : 0 \le t \le 1) \convdis (\mathsf{e}(s) : 0 \le s \le 1).
	\end{align}
	This completes the proof of Theorem~\ref{te:main1}.
\end{proof}

\begin{proof}[Proof of Theorem~\ref{te:main2}]
	We want to show that there are constants $C,c>0$ such that for all $x>0$ and $n \ge 2$
	\begin{align}
		\label{eq:repeat}
		\Pr{\He(\mN_n) > x} \le C \exp(-c x^2/n).
	\end{align}
	
	By~\eqref{eq:numvert} we know that $|\tau_n| \le 2n$. By Proposition~\ref{pro:bound} it follows that the total number $|\mN_n|$ of vertices in $\mN_n$ satisfies
	\begin{align}
		|\mN_n| &= 1 + \sum_{v \in \tau_n} (|\delta_n(v)| -1) \\
			&\le 1 + 2 \sum_{v \in \tau_n} (d_{\tau_n}^+(v) + k) \nonumber \\
			&= 1 + 2(|\tau_n|-1) + 2|\tau_n|k \nonumber \\
			&\le 4n(k+1). \nonumber
	\end{align}
	Consequently, $\He(\mN_n) \le 4n(k+1)$.
	
	Hence it suffices to show~\eqref{eq:repeat} for $x \le 4n(k+1)$. Moreover, we may always take $C$ large enough (depending on $c$) so that $C \exp(-c)>1$, implying that~\eqref{eq:repeat} is automatically fulfilled for all $0<x<\sqrt{n}$. Thus, it suffices to show the existence of $c,C>0$ such that~\eqref{eq:repeat} holds for all
	\begin{align}
		\label{eq:xas}
		\sqrt{n} \le x \le 4n(k+1).
	\end{align}

	The main theorem of~\cite{MR3077536}  establishes tail-bounds for the height of critical Galton--Watson trees conditioned on having $n$ vertices if the offspring distribution has finite variance. In \cite[Lem. 6.61, Eq. (6.41)]{MR4132643} this result and further observations from~\cite{MR3077536} were used to deduce the similar bounds for blow-ups of such trees. By~\cite{MR3335013}, \cite[Sec. 6.1.7]{MR4132643} we can view a Galton--Watson tree conditioned on having $n$ leaves as a blow-up of a different Galton--Watson tree conditioned on having $n$ vertices, allowing us to apply \cite[Lem. 6.61, Eq. (6.41)]{MR4132643} to deduce that there are constants $c_1,C_1>0$ such that for all $x>0$ and $n \ge 1$
	\begin{align}
		\label{eq:subga}
		\Pr{\He(\tau_n) > x} \le C_1 \exp( -c_1 x^2 /n).
	\end{align}
	
	Let $\epsilon>0$ be given. We will choose a suitable value for $\epsilon$ later on. Inequality~\eqref{eq:subga} entails that
	\begin{align}
		\label{eq:acombine}
		\Pr{\He(\tau_n) > \epsilon x} \le C_1 \exp( -c_1\epsilon^2 x^2 /n).
	\end{align}
	
	 A vertex of $\mN_n$ with maximal height may either be a vertex that pertains to $\tau_n$ or a vertex that lies in the surplus of some head structure. (The later case is possible since we look at directed paths. Thus a vertex with maximal height in $\mN_n$ is \emph{not} necessarily a leaf of $\mN_n$.) If such a  vertex $u$ lies in the surplus of some head structure $\delta_n(v)$ for $v \in \tau_n$, then 
	\[
		\he_{\mN_n}(u) \le \he_{\mN_n}(v) +  |\delta_n(v)|.
	\]
	Hence, if $\he_{\mN_n}(u)>x$, then $\he_{\mN_n}(v) > x/2$ or $|\delta_n(v)|>x/2$. It follows that
	\begin{multline}
		\label{eq:comcombine1}
		\Pr{\He(\tau_n) \le  \epsilon x, \He(\mN_n) >x} \\ \le \Pr{\max_{v \in \tau_n} |\delta_n(v)| >x/2} + \Pr{\He(\tau_n) \le  \epsilon x, \max_{v \in \tau_n} \he_{\mN_n}(v) >x/2}.
	\end{multline}
	By Proposition~\eqref{pro:bound} and Equation~\eqref{eq:oida1}  it follows that
	\begin{align}
		\Prb{\max_{v \in \tau_n} |\delta_n(v)| >x/2}  &\le \Pr{\Delta(\tau_n) > x/4-k}\\
													&\le O(n^{5/2}) \Pr{\xi > x/4-k}. \nonumber
	\end{align}
	Recall that by~\eqref{eq:xas} we assumed that  $\sqrt{n} \le x \le 4n(k+1)$. Hence, using Inequality~\eqref{eq:oida2} it follows that
	\begin{align}
		\label{eq:comcombine2}
		\Prb{\max_{v \in \tau_n} |\delta_n(v)| >x/2} &\le O(n^{5/2}) \exp(- \Theta(x) ) \\
														&\le \exp(-\Theta(x)) \nonumber \\
													&\le \exp(-\Theta(x^2/n)). \nonumber
	\end{align}

	In order to bound the probability for the event that simultaneously $\He(\tau_n) \le  \epsilon x$ and  $\max_{v \in \tau_n} \he_{\mN_n}(v) >x/2$, we may argue identically as for~\eqref{eq:getbound} to obtain
	\begin{align}
		\label{eq:doityyyy}
		\Pr{\He(\tau_n) \le  \epsilon x, \max_{v \in \tau_n} \he_{\mN_n}(v) >x/2} \le O(n^{5/2}) \sum_{\ell=1}^{\lfloor \epsilon x \rfloor} \Pr{\eta_1 + \ldots + \eta_\ell >x/2}.
	\end{align}
	Here $\eta_1, \eta_2, \ldots$ denote independent copies of $\eta$.
	
	Setting $\epsilon := 1/(4 \Ex{\eta})$, it follows that $x/2 - \Ex{\eta}\ell \ge x/4$ for all $1 \le \ell \le \epsilon x$. Hence, by~\eqref{eq:llld}, it follows that there is a constant $c_2 >0$ such that for all sufficiently small $\lambda$ (independent of $\ell$ and $x$) and all $1 \le \ell \le \epsilon x$
	\begin{align}
		\label{eq:llTT}
		\Pr{ \eta_1 + \ldots + \eta_\ell >x/2} \le 2 \frac{(1 + c_2 \lambda^2)^\ell }{\exp(\lambda x/4)} \le 2 \frac{(1 + c_2 \lambda^2)^{x/(4 \Ex{\eta})} }{\exp(\lambda x/4)}
	\end{align} 
	Choosing $\lambda$ small enough, it follows that there are constants $C_2, c_3>0$ such that
	\begin{align}
		\Pr{ \eta_1 + \ldots + \eta_\ell >x} &\le  C_2 \exp(-c_3 x).
	\end{align} 
	Recall that we assumed $\sqrt{n} \le x \le 4n(k+1)$ by~\eqref{eq:xas}, hence it follows from~\eqref{eq:doityyyy} that
	\begin{align}
		\label{eq:thecombine}
		\Pr{\He(\tau_n) \le  \epsilon x, \max_{v \in \tau_n} \he_{\mN_n}(v) >x/2} &\le O(n^{7/2}) \exp(-c_3 x) \\
														&\le \exp(-\Theta(x)) \nonumber \\
														&\le \exp(-\Theta(x^2/n)). \nonumber
	\end{align}
	Combining~\eqref{eq:acombine},~\eqref{eq:comcombine1},~\eqref{eq:comcombine2}  and~\eqref{eq:thecombine}, it follows that
	\begin{align}
		\Pr{\He(\mN_n) >x} \le C_3 \exp(-c_4 x^2 /n)
	\end{align}
for some constants $C_3, c_4>0$ that do not depend on $x$ or $n$.
\end{proof}

\begin{proof}[Proof of	 Thm.~\ref{te:main3}]

We define $\eta'$ analogously to $\eta$, as the length of a shortest \emph{undirected} path from the root of a uniformly selected network from $\cH[\hat{\xi}]$ to a leaf that is uniformly selected among its $\hat{\xi}$ leaves. That is, the path may cross edges in any direction, regardless of their orientation.

Lemma~\ref{le:stretch} holds analogously for undirected paths. That is, interpreting the vertices of $\tau_n$ as part of $\mG_n$, 
\begin{align}
		\label{eq:undirected}
		\sup_{v \in \tau_n} |\he_{\mG_n}(v) - \Ex{\eta'} \he_{\tau_n}(v)| \le n^{3/8}
\end{align}
holds with probability tending to $1$ as $n \to \infty$. Here $\he_{\mG_n}(v)$ denotes the graph distance $d_{\mG_n}(o,v)$ from the root vertex $o$ of $\mG_n$ to the vertex $v$.  
Hence Thm.~\ref{te:main1} and all intermediate observations in its proof hold analogously for undirected paths, that is,
\begin{align}
	\label{eq:undirfunctional}
	(b_k' n^{-1/2} \he_{\mG_n}(v_{s |\mG_n|}) : 0 \le s \le 1) \convdis (\mathsf{e}(s) : 0 \le s \le 1) 
\end{align}
for \begin{align}
	b_k' := b / \Ex{\eta'}
\end{align} as $n \to \infty$.

Given two vertices $v,w \in \tau_n$ there is a unique path $Q$ in $\tau_n$ that joins them in $\tau_n$, but there may be several different paths $P$ in $\mG_n$ that join $v$ and $w$ in $\mG_n$. However, we know that the vertices and edges of  $P$ are always entirely contained in the head-structures $(\delta_n(u))_{u \in {Q}}$.  Let $\mathrm{lca}(u,v)$ denote the lowest common ancestor of $v$ and $w$ in the tree $\tau_n$. 
The distance $d_{\tau_n}$ in the tree $\tau_n$ satisfies
\begin{align}
	d_{\tau_n}(v,w) = \he_{\tau_n}(v) + \he_{\tau_n}(w) - 2 \he_{\tau_n}(\mathrm{lca}(u,v)).
\end{align}
A similar statement holds for the distance $d_{\mG_n}$ in $\mG_n$: If $P_v$ and $P_w$ are \emph{shortest} paths in $\mG_n$ from $v$ to the root $o$ and from $w$ to $o$, we may construct a shortest path $P$ from $v$ to $w$  by following $P_v$ until we encounter for the first time a vertex $v'$ from $\delta_n(\mathrm{lca}(v,w))$, walking from $v'$ to the analogously defined vertex $w'$ from $\delta_n(\mathrm{lca}(v,w))$ by a path that lies entirely in $\delta_n(\mathrm{lca}(v,w))$, and then following $P_w$ (in its reverse direction) from $w'$ back to $w$. This entails that
\begin{align}
	\label{eq:yaaaaay}
	d_{\mG_n}(v,w) = \he_{\mG_n}(v) + \he_{\mG_n}(w) - 2 \he_{\mG_n}(\mathrm{lca}(v,w)) + R(v,w)
\end{align}
for an error term $R(v,w)$ satisfying
\begin{align}
	\label{eq:errorterm}
	|R(v,w)| \le 3 |\delta_n(\mathrm{lca}(v,w))|.
\end{align}

Recall that in Equation~\eqref{eq:specod} we constructed a special ordering
	\begin{align}
	(u_1, \ldots, u_{|\mN_n|}) := (v_1, v_{1,1}, \ldots, v_{1,s_1}, \quad \ldots, \quad   v_{|\tau_n|}, v_{{|\tau_n|},1}, \ldots, v_{{|\tau_n|},s_{|\tau_n|}})
\end{align} 
of the vertices of $\mN_n$. Here $v_1, \ldots, v_{|\tau_n|}$ is a depth-first-search ordering of the vertices of $\tau_n$, where, say, we always proceed along the left-most unvisited vertex. This way, it holds for all $1 \le i \le j \le |\mG_n|$ that 
\begin{align}
	\he_{\mG_n}(\mathrm{lca}(v_i,v_j)) = \min_{i \le r \le j} \he_{\mG_n}(v_r).
\end{align}

By Proposition~\ref{pro:bound} and Proposition~\ref{pro:max}, it follows analogously as for~\eqref{eq:hahaha} that there is a constant $C'>0$ such that \begin{align}
	\max_{v \in \tau_n} |\delta_n(v)| \le 2 (\Delta(\tau_n) + k) \le C' \log n
\end{align} with a probability that tends to $1$ as $n \to \infty$. It follows from Equation~\eqref{eq:yaaaaay} that hence there is a constant $C>0$ such that with a probability that tends to $1$ as $n \to \infty$
\begin{align}
	\label{eq:joooo}
	\left |d_{\mG_n}(u_i,u_j) - \left( \he_{\mG_n}(u_i) + \he_{\mG_n}(u_j) - 2 \min_{i \le r \le j} \he_{\mG_n}(u_r) \right) \right | \le C \log n
\end{align}
for all $1 \le i \le j \le |\mN_n|$. Multiplying  both sides of~\eqref{eq:joooo} by $b_k' n^{-1/2}$, Thm.~\ref{te:main3} now follows from~\eqref{eq:undirfunctional} and the definition~\eqref{eq:crtmetric} of the metric of the Brownian continuum random tree.
\end{proof}

\section{The asymptotic local shape}

Using the notation from Section~\ref{sec:sizebias}, we may set
\begin{align}
	(\hat{\tau}, \hat{\delta}) := (\hat{\tau}^{(\infty)}, \hat{\delta}^{(\infty)})
\end{align}
and 
\begin{align}
	\hat{\mN} := \Lambda^{-1}(\hat{\tau}, \hat{\delta}).
\end{align}
Here, by a slight abuse of notation, we use $\Lambda^{-1}$ to denote the canonical extension of the blow-up procedure $\Lambda^{-1}$ (defined for finite decorated trees) to infinite locally finite decorated trees. The infinite network $\hat{\mN}$ is the local weak limit of $\mN_n$ as $n \to \infty$:

\begin{proof}[Proof of Thm.~\ref{te:local1}]
General results for the local convergence of conditioned Galton--Watson trees~\cite{MR3164755} imply that
\begin{align}
	\tau_n \convdis \hat{\tau}
\end{align}
in the local topology. That is, for any fixed constant $\ell \ge 0$ the probability for the $\ell$-neighbourhood of the root of $\tau_n$ to assume a given shape converges as $n \to \infty$ to the probability for the $\ell$-neighbourhood of the root of $\hat{\tau}$ to assume that shape.

By Skorokhod's representation theorem we may assume that $\hat{\tau}, \tau_2, \tau_3, \ldots$ are coupled such that
\begin{align}
		\tau_n \convas \hat{\tau}.
\end{align}
This entails
\begin{align}
	(\tau_n, \delta_n) \convas (\hat{\tau}, \hat{\delta}).
\end{align}
Hence
\begin{align}
	\mN_n = \Lambda^{-1}(\tau_n, \delta_n) \convas \Lambda^{-1}(\hat{\tau}, \hat{\delta}) = \hat{\mN}.
\end{align}
This proves the local weak convergence in~\eqref{te:local2}.

Kersting~\cite[Thm. 5]{kersting2011height} described the asymptotic shape of $o(\sqrt{n})$-neighbourhoods of critical Galton--Watson trees conditioned on having $n$ vertices (if the offspring distribution has finite variance). Using a transformation from~\cite{MR3335013,MR1284403}, it follows that this also holds for Galton--Watson trees conditioned on having $n$ leaves, yielding
\begin{align}
	d_{\textsc{TV}}(U_{\ell_n}(\tau_n), U_{\ell_n}(\hat{\tau}) ) \to 0
\end{align}
for any sequence $\ell_n$ of positive integers satisfying $\ell_n= o(\sqrt{n})$. Consequently, we may assume that $\hat{\tau}, \tau_1, \tau_2, \ldots$ are coupled such that
\begin{align}
	U_{\ell_n}(\tau_n) = U_{\ell_n}(\hat{\tau})
\end{align}
holds with a probability that tends to $1$ as $n \to \infty$.  Consequently, 
we may construct the decorations in such a way that
\begin{align}
	(U_{\ell_n}(\tau_n), (\delta_n(v))_{v \in U_{\ell_n}(\tau_n)})  = (U_{\ell_n}(\hat{\tau}), (\hat{\delta}(v))_{v \in U_{\ell_n}(\hat{\tau})} )
\end{align}
holds with a probability tending to $1$ as $n \to \infty$. This entails that
\begin{align}
	U_{\ell_n}(\mN_n) = U_{\ell_n}(\Lambda^{-1}(\tau_n, \delta_n) ) = U_{\ell_n}(\Lambda^{-1}(\hat{\tau}, \hat{\delta}) )  =  U_{\ell_n}(\hat{\mN})
\end{align}
again holds with probability tending to $1$ as $n \to \infty$. This implies Equation~\eqref{eq:sqqrt} and completes the proof.
\end{proof}

The surplus vertices of the head structures influence the location of a uniformly selected vertex. To take them into account, we form the tree $\cT_n$ by colouring the vertices of $\tau_n$ blue and adding to each vertex $v \in \tau_n$  additional $S(\delta_n(v))$ red children. We define $\cT$ in same way by adding the surplus vertices of $(\tau, \delta)$ as red children at the appropriate vertices, making $\cT$ a two-type Galton--Watson tree with offspring distribution $(\xi, \kappa)$, with $\kappa$ defined as in Lemma~\ref{le:spread}. (This makes $\xi$ and $\kappa$ dependent on each other.) A uniformly selected vertex $u_n$ of $\mN_n$ hence corresponds to a uniformly selected vertex of $\cT_n$. 

By a general result \cite[Thm. 3]{JOTMULT} for conditioned multi-type Galton--Watson trees, there is a random $2$-type tree $\hat{\cT}^*$ such that
\begin{align}
	\label{eq:cccome}
	(\cT_n,u_n) \convdis \hat{\cT}^*
\end{align}
in the local topology as $n \to \infty$. The tree $\hat{\cT}^*$ has a marked vertex with a finite number of descendants, and an infinite number of ancestors, each having a random finite number of descendants in total. (This agrees with the intuition that most vertices of $\cT_n$ are far from the root and have few descendants.) We may assign random decorations $\hat{\delta}^*(v)$, $v \in  \hat{\cT}^*$ in the same way as before (for $\delta$, $\delta_n$, and $\hat{\delta}$), allowing us to define the infinite network
\begin{align}
	\hat{\mN}^* := \Lambda^{-1}(\hat{\cT}^*, \hat{\delta}^*).
\end{align}
The network $\hat{\mN}^*$ describes the asymptotic shape of the vicinity a random vertex of $\mN_n$, similarly as $\hat{\mN}$ describes the vicinity of the fixed root vertex of $\mN_n$:

\begin{proof}[Proof of Thm.~\ref{te:local2}]
	Applying Skorokhod's representation theorem to~\eqref{eq:cccome}, we may assume that $\hat{\cT}^*, (\cT_2,u_2) , (\cT_3, u_3), \ldots$ are coupled such that
	\begin{align}
		(\cT_n,u_n) \convas \hat{\cT}^*.
	\end{align}
	This entails
	\begin{align}
		((\cT_n,u_n), \delta_n) \convas (\hat{\cT}^*, \hat{\delta}^*).
	\end{align}
	Hence
	\begin{align}
		(\mN_n,u_n) = (\Lambda^{-1}(\cT_n, \delta_n), u_n) \convas \hat{\mN}^*.
	\end{align}
	This proves the local weak convergence in~\eqref{eq:annealed}. The proof of the extension~\eqref{eq:sqqrt2} is entirely analogous to the proof of the corresponding statement~\eqref{eq:sqqrt} for the fixed root, by building on a limit~\cite[Eq. (6.44)]{graphonaccepted} for the $o(\sqrt{n})$-neighbourhood of random points in conditioned sesqui-type trees.
	
	In order to prove the quenched convergence in~\eqref{eq:quenched}, note that stating~\eqref{eq:quenched} to hold for all $\ell \ge 1$ and all vertex-marked graphs $G$ is equivalent to stating that the random probability measure $\mathfrak{L}( (\mN_n, u_n) \mid \mN_n)$ given by the uniform measure on the $|\mN_n|$ many vertex-marked versions of $\mN_n$ satisfies
	\begin{align}
		\label{eq:toshoooow}
		\mathfrak{L}( (\mN_n, u_n) \mid \mN_n) \convdis \mathfrak{L}(\hat{\mN}^*),
	\end{align}
	with $\mathfrak{L}(\hat{\mN}^*)$ denoting the deterministic law of $\hat{\mN}^*$. Now,  \cite[Thm. 3]{JOTMULT}  ensures such a quenched limit for the uniform measure $\mathfrak{L}( (\cT_n, u_n) \mid \cT_n) $ on the $|\cT_n|$-many ($=|\mN_n|$ many) vertex marked versions of $\cT_n$, with the deterministic limit measure given by the law $\mathfrak{L}(\hat{\cT}^*)$ of $\hat{\cT}^*$. That is,
	\begin{align}
			\mathfrak{L}( (\cT_n, u_n) \mid \cT_n) \convdis \mathfrak{L}(\hat{\cT}^*).
	\end{align}
	 Using the Chernoff bounds, this implies such a limit for the decorated versions:
	 \begin{align}
	 	\label{eq:almost}
	 	\mathfrak{L}( ((\cT_n, u_n),\delta_n) \mid \cT_n) \convdis \mathfrak{L}(\hat{\cT}^*, \hat{\delta}^*).
	 \end{align}
	The limit~\eqref{eq:toshoooow} now follows from~\eqref{eq:almost} by applying the blow-up procedure $\Lambda^{-1}$ and the continuous mapping theorem. This completes the proof.
\end{proof}

\bibliographystyle{abbrv}
\bibliography{blank}

\end{document}